\documentclass{amsart}
\usepackage[latin1]{inputenc}
\usepackage[T1]{fontenc}
\usepackage{amssymb, amsmath, color, textcomp, url,tikz, mathtools}
\usetikzlibrary{matrix,arrows}
\usepackage{courier}
\usetikzlibrary{fit}
\usetikzlibrary{positioning}
\usepackage[labelformat=empty]{caption}

\usepackage[normalem]{ulem}
\usepackage{soul}
\usepackage{mdwlist}

\usepackage{tikz-cd}

\RequirePackage{ifpdf}
\ifpdf
\usepackage[pdftex]{hyperref}
\else
\usepackage[hypertex]{hyperref}
\fi

\usepackage{enumerate,xspace}
\usepackage{chngcntr, csquotes} 
\theoremstyle{plain}
\newtheorem{theorem}{Theorem}[section]
\newtheorem{cor}[theorem]{Corollary}
\newtheorem{prop}[theorem]{Proposition}
\newtheorem{lemma}[theorem]{Lemma}

\newcounter{proofcount}
\AtBeginEnvironment{proof}{\stepcounter{proofcount}}

\newtheorem*{claim*}{Claim}
\makeatletter                  
\@addtoreset{claim}{proofcount}
\makeatother                   

\usepackage{pdfpages}

\newtheorem{thm}{Theorem}

\theoremstyle{definition}
\newtheorem{remark}[theorem]{Remark}
\newtheorem{fact}[theorem]{Fact}
\newtheorem{definition}[theorem]{Definition}
\newtheorem{example}[theorem]{Example}

\newcounter{substep}
\def\thesubstep{\arabic{substep}}

\newenvironment{substeps}{%
	\refstepcounter{substep}\noindent{(\thesubstep)}}%
{\em}

\newcounter{subsubstep}
\def\thesubsubstep{\arabic{subsubstep}}

\newenvironment{subsubsteps}[1]{%
	\refstepcounter{subsubstep}\noindent{(#1.\thesubsubstep)\ }\ }%
{\em}

\newcounter{substepC}
\def\thesubstepC{\arabic{substepC}}

{\em}

\newcounter{subsubstepC}[substepC]
\def\thesubsubstepC{\arabic{subsubstepC}}

{\em}

\newcommand{\nc}{\newcommand}

\nc{\Z}{\mathbb{Z}}
\nc{\Q}{\mathbb{Q}}
\nc{\N}{\mathbb{N}}
\nc{\F}{\mathbb{F}}
\nc{\UU}{\mathbb{U}}
\nc{\C}{\mathbb{C}}

\nc{\M}{\mathcal{M}}
\nc{\R}{\mathcal{R}}
\nc{\A}{\mathcal{A}}
\nc{\B}{\mathcal{B}}
\nc\LL{\mathcal L}
\nc\II{\mathcal I}
\nc\E{\mathcal E}

\nc{\stt}{\operatorname{St}}
\nc{\stab}{\operatorname{Stab}}
\nc{\GO}[1]{G_{#1}^{00}}
\nc{\sbgp}[1]{\langle\xspace {#1}\xspace\rangle}
\nc{\Conn}[1]{\langle\xspace {X}\xspace\rangle^{00}_{#1}}
\nc{\band}[1]{\bar d_{\mathcal{#1}}}
\nc\Def{\operatorname{Def}}

\nc{\dcl}{\operatorname{dcl}}

\nc{\acl}{\operatorname{acl}}

\nc{\nf}[1]{_{\mid {#1}}}
\nc{\restr}[1]{\xspace_{\upharpoonright {#1}}}

\nc\inv{ ^{-1}}

\nc{\tp}{\operatorname{tp}}
\nc\Spec{S^\mathrm{t}}
\nc\HS{S^\mathrm{h}}
\nc\U{\operatorname{U}}
\nc{\cf}{\text{cf.\,}}
\nc{\eg}{\text{e.g. }}

\nc{\InvR}[1]{S_{#1}(\bar{\R})^{\rm inv}_{R}}
\nc{\CohR}[1]{S_{#1}(\bar{\R})^{\rm fs}_{R}}
\nc{\InvRt}[1]{S^{\rm t}_{#1}(\bar{\R})^{\rm inv}_{R}}

\nc{\Inv}[1]{S_{#1}(\bar{\M})^{\rm inv}_{M}}
\nc{\Invh}[1]{S^{\rm h}_{#1}(\bar{\M})^{\rm inv}_{M}}
\nc{\Coh}[1]{S_{#1}(\bar{M})^{\rm fs}_{M}}
\nc{\Ext}[1]{S_{#1}(\M)^{\rm ext}}

\nc\bM{\overline{M}}

\def\Ind#1#2{#1\setbox0=\hbox{$#1x$}\kern\wd0\hbox to
	0pt{\hss$#1\mid$\hss} \lower.9\ht0\hbox to
	0pt{\hss$#1\smile$\hss}\kern\wd0}
\def\Notind#1#2{#1\setbox0=\hbox{$#1x$}\kern\wd0\hbox to
	0pt{\mathchardef\nn="0236\hss$#1\nn$\kern1.4\wd0\hss}\hbox to
	0pt{\hss$#1\mid$\hss}\lower.9\ht0 \hbox to
	0pt{\hss$#1\smile$\hss}\kern\wd0}

\title[]{Ellis enveloping semigroups in real closed fields}

\address{Departamento de \'Algebra, Geometr\'ia y Topolog\'ia; Facultad de Matem\'aticas;
	Universidad Complutense de Madrid; 28040 Madrid, Spain}

\email{ebaro@ucm.es}

\email{dpalacin@ucm.es}

\date{\today}

\author{El\'ias Baro and Daniel Palac\'in}

\thanks{Both authors are supported by Spanish STRANO  PID2021-122752NB-I00 and Grupos UCM 910444. }

\begin{document}
	\maketitle
	\begin{abstract} 
		We introduce the Boolean algebra of $d$-semialgebraic (more generally, $d$-definable) sets and prove that its Stone space is naturally isomorphic to the Ellis enveloping semigroup of the Stone space of the Boolean algebra of semialgebraic (definable) sets. For definably connected o-minimal groups, we prove that this family agrees with the one of externally definable sets in the one-dimensional case. Nonetheless, we prove that in general these two families differ, even in the semialgebraic case over the real algebraic numbers.  On the other hand, in the semialgebraic case we characterise real semialgebraic functions representing Boolean combinations of $d$-semialgebraic sets.  
	\end{abstract}

	\section{Introduction}
	
	A recurrent matter in real algebraic geometry is attempting to characterize some properties of an algebraic or semialgebraic set defined over a parameter set using a simpler set of parameters. For instance, in \cite{PR20} (see also \cite{FGa}) the authors prove that any affine algebraic variety defined over the real field is homeomorphic to a variety defined over the field $\mathbb{R}_{\text{alg}}$ of real algebraic numbers. They do not achieve this by applying Tarski's transfer principle, but instead by producing a deformation of the coefficients. To accomplish this, they consider sets of the form $Z\cap \mathbb{R}^n_{\text{alg}}$ where  $Z\subset \mathbb{R}^n$ is a semialgebraic set defined over $\mathbb{R}$. These sets are known as \emph{externally semialgebraic} and constitute the central topic of our paper. For example, the set
	$$X_\pi:=\{(x,y)\in  \mathbb{R}_{\rm alg}^2 \ | \ y<\pi \cdot x\}.$$
	More generally, one can define the notion of externally semialgebraic set for an arbitrary real closed field $R_1$. Namely, a subset $X$ of $R_1^n$ is {\em externally semialgebraic} if there exists a larger real closed field $R_2$ and a semialgebraic subset $Z\subset R_2^n$ such that $X=Z\cap R_1^n$.
	
	The problem that we address in this paper is in a certain sense similar to the one described above for real closed fields. We analyse whether any externally semialgebraic subset $X$ of $R_1^n$ can be described in a specific way using a semialgebraic subset from a larger real closed field $R_2$. More precisely, whether any externally semialgebraic set $X\subset R_1^n$ is {\em $d$-semialgebraic}, that is, if there exists a semialgebraic subset $Z\subset R_2^n$ defined over $R_1$ and a point $a\in R_2^n$ such that $X=(Z-a)\cap R_1^n$. 
	
 The objective is to prove for $n\ge 2$ that the collection of sets formed by  Boolean combinations of $d$-semialgebraic sets of  $R_1^n$ is different from the collection of externally semialgebraic sets of $R_1^n$. In fact, we show that $X_\pi$ above is not a Boolean combination of $d$-semialgebraic sets.  This is a reasonable statement, since multiplication by $\pi$ should not be possible to code just via  multiplication by real algebraic numbers and translations by transcendental numbers. 
 
 Although the problem can be stated in terms of classical semialgebraic geometry, the solution relies in model-theoretic tools. So, since the above notions adapt to {\em definable} sets, we shall not only focus on real closed fields and semialgebraic sets, but on arbitrary (o-minimal) structures and definable sets, which are precisely the semialgebraic ones in the real algebraic setting. It is worth noticing that the study of {\em externally definable} sets in model theory is a recurrent topic. This goes back to \cite{BP98} where the authors prove that externally semialgebraic subsets of a real closed field are precisely finite unions of convex subsets, a crucial property in our work. 
	
	The introduction of $d$-definable sets and the motivation to show that these two collections of sets are not equal comes from the theory of Ellis enveloping semigroups.	
	In his seminal work \cite{lN09}, Newelski made several connections between model theory and topological dynamics, which has become an active area of research over the past decade, see for instance \cite{CS18,GPP15,J15,K17,P13,PY16} and \cite{YL15} for some further reading. 
	
	We briefly recall Newelski's construction. Let $G$ be a definable group in a structure $M$. The group  acts naturally by homeomorphisms on the compact Hausdorff space of types $S_G(M)$ concentrated in $G$ (that is, space of ultrafilters of definable subsets of $G$). A classical construction due to Ellis \cite{rE69} permits to associate to this action the so-called \emph{Ellis enveloping semigroup} $(E(S_G(M)),\circ)$, {\it i.e.} a  compact Hausdorff topological space equipped with a semigroup operation which is continuous in the first coordinate. In \cite[pp. 68-69]{lN09}, Newelski gives two abstract conditions to identify the Ellis enveloping semigroup $E(S_G(M))$ with the Stone space $S_G(M)$. Nonetheless, as he points out, these conditions may fail even when  $G$ is the additive group of algebraic real numbers and $M=\mathbb R_{\rm alg}$. 	
	
	Later on, Newelski \cite{lN12} observed that to understand the dynamics of the group action is more convenient to work with externally definable sets, and consider the action of $G$ on the Stone space $S^{\rm ext}_G(M)$ of ultrafilters of externally definable sets of $M$. In this situation, the semigroup $(E(S^{\rm ext}_G(M)),\circ)$ turns out to be $(S^{\rm ext}_G(M),*)$ itself, with a well-known model theoretic operation $*$ called \emph{coheir product}.
	
	Whilst considering the action on $S^{\rm ext}_G(M)$ eases the model-theoretic treatment of the Ellis semigroup,  it seems also natural to ask whether $E(S_G^{\rm ext}(M)) \simeq S_G^{\rm ext}(M)$ are naturally isomorphic to $E(S_G(M))$ as Ellis semigroups. 	
	To the best of our knowledge, there is no example in the literature showing they are different. In fact, the question has been completely neglected, apart from the aforementioned remark due to Newelski. 
	 Hence, the primary objective of our paper is to demonstrate that $S_G^{\rm ext}(M)$ and $E(S_G(M))$ are \emph{not} naturally isomorphic in general.
	
	In order to achieve this objective, in Sections \ref{sec:2} and \ref{sec:3} we revisit and refine certain aspects of the framework established by Newelski in his groundbreaking works \cite{lN09,lN12} and \cite{lN14}, by studying Ellis semigroups of Stone spaces. As a consequence we obtain the following characterization:
	
\begin{thm}[Corollary \ref {C:Equiv-Ellis-ext}]  Let $M$ be an arbitrary structure.	The Ellis semigroups $S_G^{\rm ext}(M)$ and $E(S_G(M))$ are canonically isomorphic if and only if every externally definable subset of $G(M)$ is a (positive) Boolean combination of $d$-definable sets.
\end{thm}

	In Section \ref{sec:4} we then focus in the o-minimal context. 	Note that if the universe of $M$ is the (Dedekind complete) real field $\mathbb{R}$ then $S^{\rm ext}_G(M)=S_G(M)$. Therefore, to find an example where $S^{\rm ext}_G(M)$ is not $E(S_G(M))$ it is natural to consider either $M$ an $\aleph_0$-saturated o-minimal structure, or $M=\mathbb{R}_{\text{alg}}$. We prove the following result:
	
\begin{thm}[Theorem \ref{T:1dim} and \ref{main}, Corollary \ref{C:Alg}] Let $M$ be an o-minimal structure and let $G$ be a definable group.  The following hold:   
    \begin{enumerate}[$i)$]
    	\item If $\dim(G)=1$, then $(S^{\rm ext}_G(M),*)$ and $(E(S_G(M)),\circ)$ are naturally isomorphic as Ellis semigroups.

	   	\item If $M$ is an $\aleph_0$-saturated expansion of a real closed field and $G=(M^2,+)$, then  $(S^{\rm ext}_G(M),*)$ and $(E(S_G(M)),\circ)$ are not naturally isomorphic as Ellis semigroups.
		    	
    	\item If $M$ is an expansion of the field $\mathbb{R}_{\text{alg}}$ and $G=(\mathbb{R}^2_{\text{alg}},+)$, then  $(S^{\rm ext}_G(M),*)$ and $(E(S_G(M)),\circ)$ are not naturally isomorphic as Ellis semigroups.
    \end{enumerate}
\end{thm}
In addition, in Remark \ref{defcompact} we also construct an example of a definably compact group $G$ for which $(S^{\rm ext}_G(M),*)$ and $(E(S_G(M)),\circ)$ are not naturally isomorphic as Ellis semigroups.

Finally, before finishing the introduction, we make some comments on the proof of the last theorem. For the proof of $iii)$ we establish a general statement (Proposition \ref{prop:realalg}) that enables us to characterize the real semialgebraic functions representing Boolean combinations of $d$-semialgebraic sets. Using this proposition, we deduce that $X_\pi$ mentioned above is not a Boolean combination of $d$-semialgebraic sets. Furthermore, this proposition allows us to demonstrate that many other externally semialgebraic sets (intuitively, any externally semialgebraic set described by a semialgebraic function using the product by transcendental numbers) are also not expressible as such Boolean combinations.

The situation in $ii)$ is quite different. We prove that a specific  externally definable set is not a Boolean combination of $d$-definable sets. Although the general technique can be applied to other externally definable sets, the proof depends on certain bounds calculated for the specific example we are considering. Moreover, we show there are examples of externally definable sets that at first sight could seem $d$-definable, but are not (see Example \ref{extddef}).

\subsection*{Acknowledgments} We thank the referee for the careful reading of the paper and all suggestions made. In particular, we wish to express our gratitude for pointing us the work of Adam Malinowski and Ludomir Newelski, concerning the alternative representation of the Ellis semigroup as Stone spaces in \cite{aM23,MN23}.

	\section{Set-up on Ellis semigroups of Stone spaces}\label{sec:2}
	
	The goal of the section is to give a general framework to study the Ellis enveloping of some spaces of types. It will be convenient to regard these as Stone spaces. So, we fix a very general set-up in terms of Boolean algebras following the approach of Newelski \cite[Section 1]{lN12} and \cite[Section 1]{lN14}, see also \cite[Section 3]{gC21}.
	
	\begin{definition}
		An {\em Ellis semigroup} is a semigroup $(E,\cdot)$ which is a compact Hausdorff topological space such that $\cdot$ is continuous in the first coordinate, that is, for each $y\in E$ the map $x\mapsto x\cdot y$ is continuous. 
	\end{definition}

	A natural example of Ellis semigroup is the Ellis enveloping semigroup. We recall briefly its construction, see \cite{rE69}. Suppose that $S$ is a $G$-flow, {\it i.e.} let $G$ be a group with the discrete topology acting on a compact Hausdorff space $S$ by homeomorphisms. Thus the action is given by a (continuous) homomorphism $G\to \mathrm{Homeo}(S)\le S^S$ with $g\mapsto \ell_g$. The {\em Ellis enveloping semigroup}  $E(S)$ is the closure of $\{\ell_g\}_{g\in G}$ in $S^S$, where $S^S$ is equipped with the product topology.
	
	\begin{fact}[Ellis]
The pair  $(E(S),\circ)$ is an Ellis semigroup, where $\circ$ denotes the composition of functions.
	\end{fact}

	Now, fix a group $G$ and let $\A\subset \mathcal P(G)$ be a Boolean algebra. Consider the \emph{Stone space} $S(\A)$ of $\A$, which recall is the set of ultrafilters on $\A$. This is a compact Hausdorff totally disconnected topological space with the Stone topology, {\it i.e.} the topology generated by the family of subsets of the form $[X]=\{p\in S(\mathcal A) \ | \ X\in p\}$ where $X$ is an arbitrary element of $\A$. For $g\in G$, we write $p_g^\A$ to denote the principal ultrafilter $\{X\in\A \ |\ g\in X\}$ associated to $g$.
	
	Suppose that $\A$ is left-invariant, that is, it is closed under left-translation by elements of $G$. For $g\in G$ define  $\ell^{\A}_g:S(\A)\rightarrow S(\A)$ as 
	$$
	\ell^{\A}_g(q):=\{X\in \A \ | \ g^{-1}X\in q\}.$$ The inverse of $\ell^{\A}_g$ is clearly  $\ell^{\A}_{g^{-1}}$. Thus, each map $\ell_g^\A$ is a homeomorphism since for $X\in \A$ we have 
		\[(\ell_g^\A)\inv \left([X] \right) = \left\{ q\in S(\A) \ | \ X\in \ell_g^\A(q) \right\} = [g\inv X],
		\]
	which is a basic open subset of $S(\A)$.  Therefore $S(\A)$ is a $G$-flow. 		
	
	For an ultrafilter $p\in S(\A)$, consider the map
	\[
	d_p : \A\to \mathcal P(G), \ X\mapsto d_p X := \left\{ g\in G \ | \ g\inv X\in p \right\}.
	\]
	This is a homomorphism of Boolean algebras which preserves left-translation. Note that {\em a priori} there is no reason why the image of $d_p$ is contained in $\A$. 
	
	We recall the following definition from \cite{lN14}.
	
	\begin{definition}
	Let $\A\subset \mathcal P(G)$ be a left-invariant Boolean algebra. We say that $\A$ is {\em $d$-closed} if $d_qX\in \A$ for every $q\in S(\A)$ and $X\in \A$.
	\end{definition}

	\begin{remark}\label{R:d-closed}
	If a left-invariant Boolean algebra $\A$ is $d$-closed, then $\A$ is necessarily right-invariant, {\it i.e.} it is closed under right-translation by elements of $G$. Indeed, in general, given $X\in \A$ and $h\in G$ we have that
	\begin{equation}\label{eq:0}
	Xh = \left\{ g\in G \ | \ h^{-1}\in g^{-1}X \right\} = \left\{ g\in G \ | \ g^{-1}X \in p_{h^{-1}}^\A \right\} = d_{p_{h^{-1}}^\A} X.
	\end{equation}
	Thus, assuming that $\A$ is $d$-closed, we see that $Xh\in\A$.
	\end{remark}

We introduce the following notion, which also appears in \cite[Section 2.3]{aM23} in an equivalent form (see Remark \ref{rmk:Adleft} below).

\begin{definition}
Given a left-invariant Boolean subalgebra  $\A\subset \mathcal P(G)$ we define $\A^d\subset \mathcal P(G)$ to be the Boolean subalgebra generated by the sets of the form $d_p X$ for $p\in S(\A)$ and $X\in \A$.
\end{definition}

Note that by taking $h=1_G$ in the expression (\ref{eq:0}) of Remark \ref{R:d-closed} we immediately get that $\A\subset \A^d$, so $\A=\A^d$ whenever $\A$ is $d$-closed. Furthermore, we also have the following.

\begin{remark}\label{rmk:Adleft} Let $\A$ be a left-invariant Boolean algebra.
	\begin{enumerate}
		\item  The Boolean algebra $\A^d$ is left-invariant because $hd_qY=d_qhY$ for any $q\in S(\A)$, $h\in G$ and $Y\in \A$.
		\item If $\A\subset \B$ are two left-invariant Boolean algebras such that $\B$ is $d$-closed, then $\A^d\subset \B^d$. Indeed, given some set $X\in \A\subset \B$ and some $p\in S(\A)$, choose some ultrafilter $q\in S(\B)$ such that $q_{|\A} = p$. It then follows that
		\[
		g\in d_p X \ \Leftrightarrow \ g^{-1}X \in p \ \Leftrightarrow \ g^{-1}X \in q \ \Leftrightarrow \ g\in d_q X, 
		\]
		showing that  $d_pX = d_qX\in \B^d$, as required. In particular, if  $\B$ is $d$-closed then $\A^d\subset \B$.
	\end{enumerate}
It follows from (1) and (2) that $\A^d$ is the smallest left-invariant $d$-closed Boolean algebra containing $\A$. This results has also been proved independently in \cite[Section 2.2]{aM23}. Indeed, Malinowski defines $\A^d$ as the minimal $d$-closed left-invariant subalgebra of $\mathcal P(G)$ and then proves in \cite[Fact 2.22]{aM23} that it is precisely the Boolean algebra generated by the sets $d_p X$ for $p\in S(\A)$ and $X\in \A$.
\end{remark}

Next, we will also see that $\A^d$ is itself $d$-closed but before it is convenient to introduce some further notation. 

The following is a mere translation of  \cite[Lemma 1.5(1)]{lN12} to this general setting (cf. \cite[Lemma 3.10]{gC21}). We give the proof for the sake of completeness. 
	\begin{fact}\label{F:lmap}
	Let $\A\subset \B$ be two left-invariant Boolean subalgebras of $\mathcal P(G)$ such that $\A^d\subset \B$. For $p\in S(\B)$, the map
	\[
	\ell_p^\A : S(\A)\to S(\A), \ q\mapsto \ell_p^\A(q):= \left\{ X\in \A \ | \ d_q X\in p \right\}
	\]
	is well-defined and is the limit of $(\ell^{\A}_g)_{g\in G}$ with respect to the ultrafilter $p$ in the pointwise convergence topology in the space of functions from $S(\A)$ to $S(\A)$.
	\end{fact}
	\begin{proof}
 	Observe first that the set $\ell_p^\A(q)$ is well-defined since $d_qX\in \A^d\subset \B$ and $p\in S(\B)$. Furthermore, the set $\ell_p^\A(q)$ is an ultrafilter on $\A$. Hence, the map is well-defined. Also, for $g\in G$ observe that $\ell_g^{\A}$ coincides with the map associated to the principal ultrafilter $p_g^\B$ because 
 	\[
 	\ell_g^\A(q) = \left\{ X\in \A \ | \ d_q X\in p_g^\B  \right\} =\left\{ X\in \A \ | \ g\in d_q X \right\} =  \left\{ X\in \A \ | \ g\inv X \in q \right\}.
 	\]
 	  Finally,  for an ultrafilter $p\in S(\B)$ and $X\in\A$ we have that 
 	\begin{align*}
 	\lim_{g\to p}\ell^{\A}_g(q)\in [X] \ & \Leftrightarrow \ \left\{g\in G \ | \ \ell^{\A}_g(q)\in [X]\right\} \in p \ \Leftrightarrow \ \left\{g\in G \ | \ g\inv X\in q \right\} \in p \\ \ & \Leftrightarrow  \ d_q X \in p \ \Leftrightarrow \ X\in \ell^{\A}_p(q)  \ \Leftrightarrow \ \ell^{\A}_p(q)\in [X],
 	\end{align*} 
 	which shows that $\lim_{g\to p}\ell^{\A}_g = \ell_p^{\A}$. 
	\end{proof}
	
Note that for a left-invariant Boolean algebra $\A$, Fact \ref{F:lmap} provides a description of the Ellis semigroup $(E(S(\A)),\circ )$ of the $G$-flow $S(\A)$. Indeed, for any $\B$ left-invariant Boolean subalgebra of $\mathcal P(G)$ with $\A^d\subset \B$ we have that			
	\[
	E(S(\A)) = \left\{ \ell_p^\A \ | \  p\in S(\B) \right\} = \left\{ \ell_p^\A \ | \  p\in S(\A^d) \right\},
	\]
	where the second equality holds by considering the restriction map $S(\B) \twoheadrightarrow S(\A^d)$. Henceforth, to easier notation, we omit to write the superscript $\A$ in $\ell_p^\A$ when there is no possible confusion.

	In addition, Newelski proved \cite[Proposition 2.4]{lN14} (see also Proposition \ref{P:Epi} below): 
	\begin{fact}\label{F:Ellis-Newelski}
	Let $\A\subset\mathcal P(G)$ be a left-invariant $d$-closed Boolean sub\-algebra. We have that $(S(\A),*)$ is an Ellis semigroup, where $*$ is defined as $p*q:=\ell_p(q)$. Furthermore, the map 
	\[
	\Lambda : S(\A) \to E(S(\A)) , \ p\mapsto \ell_p,
	\]
	is an isomorphism of Ellis semigroups. 
	\end{fact}

Now we see that $\A^d$ is $d$-closed. So, the result above will hold true for $\A^d$. 

\begin{lemma}\label{L:d-closed}
	The Boolean algebra $\A^d$ is the smallest left-invariant $d$-closed Boolean algebra containing  $\A$.
\end{lemma}
\begin{proof} By Remark \ref{rmk:Adleft}, it remains to prove that $\A^d$ is $d$-closed. Since for each $p\in S(\A^d)$ the map $d_p:\A^d\to \mathcal P(G)$ is a homomorphism of Boolean algebras, it suffices to show that $d_p X \in \A^d$ for arbitrary $p\in S(\A^d)$ and $X=d_q Y$ with $q\in S(A)$ and $Y\in \A$. In that case, and since $\ell^\A_p(q)\in S(\A)$ by Fact \ref{F:lmap}, we have
	\begin{align*}
	g\in d_pX  = d_p(d_qY) \ & \Leftrightarrow \ g\inv d_qY\in p \ \Leftrightarrow \ d_q g\inv Y \in p \\ \ & \Leftrightarrow \ g\inv Y \in \ell_p^\A(q) \ \Leftrightarrow \ g\in d_{\ell_p^\A(q)} Y.
	\end{align*}
Hence, we deduce that $d_pX=d_{\ell_p^\A(q)}Y$ belongs to $\A^d$, as desired.
\end{proof}

As a consequence, it follows from Fact \ref{F:Ellis-Newelski} that $S(\A^d)$ is naturally isomorphic to $E(S(\A^d))$ as Ellis semigroups. Furthermore, an inspection of the original proof in \cite[Proposition 2.4]{lN14} yields the following. In fact, it explicitly appears in \cite[Theorem 2.24]{aM23} where it is credited to Newelski. 
	
	\begin{prop}\label{P:Epi}
	Let $\A\subset\B$ be two  left-invariant Boolean subalgebras of $\mathcal P(G)$ such that  $\B$ is $d$-closed. The map 
	\[
	\Lambda : S(\B) \to E(S(\A)), \ p\mapsto \Lambda(p)=\ell_p^\A
	\] 
	is an epimorphism of Ellis semigroups. Furthermore, it is an isomorphism when $\A^d = \B$.
	\end{prop}
	\begin{proof} We first prove that $\Lambda: S(\B)\to S(\A)^{S(\A)}$ is continuous, where $S(\A)^{S(\A)}$ is equipped with the product topology. It suffices to show that the map $S(\B)\rightarrow S(\A)$ given by $p\mapsto \ell_p^\A (q)$ is continuous for every $q\in S(\A)$. So, let $X\in \A$, consider the basic open set $[X]$ and set $U\subset S(\B)$ to be its preimage  under this map. Note that for $p\in S(\B)$ we have 
	\[
	p\in U \ \Leftrightarrow \ \ell_p^\A (q)\in [X]  \ \Leftrightarrow \ X\in \ell_p^\A (q) \ \Leftrightarrow \ d_qX \in p \ \Leftrightarrow \ p\in [d_qX]  .
	\]
	Since $d_qX\in \A^d\subset B$, we deduce that $U$ is a (basic) open set, showing that the map $p\mapsto \ell_p^\A (q)$ is continuous, as desired.
	
	Once we have seen that $\Lambda: S(\B)\to S(\A)^{S(\A)}$ is continuous, we obtain that its image $\mathrm{im}(\Lambda)$ is a closed subset of $S(\A)^{S(\A)}$. So, as clearly $\ell_g\in \mathrm{im}(\Lambda)$ for $g\in G$, we deduce that $E(S(\A))\subset \mathrm{im}(\Lambda)$. Moreover, we get the equality by Fact \ref{F:lmap}. Therefore, we have shown that $\Lambda : S(\B) \to E(S(\A))$ is an epimorphism of Ellis semigroups.
	
 	For the second part of the statement it is enough to prove that $\Lambda$ is injective, as $S(\A^d)$ is compact and $E(S(\A))$ Hausdorff. Let $p_1,p_2\in S(\A^d)$ be  two distinct ultrafilters. So, there is some set $Y\in \A^d$ such that $Y\in p_1$ but $Y\not \in p_2$, which we may assume to be of the form $d_qX$ for some $X\in\A$ and $q\in S(\A)$. Note that
 	\[
 	Y\in p_i \ \Leftrightarrow \ d_qX\in p_i \ \Leftrightarrow \ X\in \ell_{p_i}^\A(q). 
 	\]
 	So, we obtain $\ell_{p_1}^\A \neq \ell_{p_2}^\A$ and hence the map $\Lambda$ is injective. 
 	\end{proof}
 	
	We finish the section by pointing out that $\A^d$ is the unique Boolean subalgebra that yields a natural isomorphism between $(S(\A^d),*)$ and $(E(S(\A)),\circ)$.
 	\begin{cor}\label{C:Equivalence}
 	Let $\A\subset \B$ be two left-invariant Boolean subalgebras of $\mathcal P(G)$ such that  $\B$ is $d$-closed. The following are equivalent:
 	\begin{enumerate}
 		\item The natural restriction $r:S(\B)\twoheadrightarrow S(\A^d)$ is an homeomorphism.
 		\item The map $\Lambda:S(\B)\to E(S(\A))$, given by $\Lambda(p)=\ell_p^\A$, is an isomorphism of Ellis semigroups.
 		\item It holds that $\B=\A^d$.
 	\end{enumerate}
 	\end{cor}
  	\begin{proof} Notice first that the natural restriction $r:p\mapsto \{X\in\A^d \ | \ X\in p\}$ is well-defined, continuous and surjective. If in addition $r$ is an homeomorphism, since any basic open set in $S(\B)$ is also closed, we deduce that such a basic open set is a finite union of basic open subsets of $S(\B)$ of the form $[X]$ for $X\in \A^d$. In particular, it follows that $\B=\A^d$. Hence, we obtain the equivalence between (1) and (3). Also, condition (3) implies (2) by Proposition \ref{P:Epi}. To prove that (2) implies (1), note that $\Lambda :S(\B)\to E(S(\A))$ is $\Lambda(p) = \ell_{r(p)}^\A$. Indeed, for $X\in \A$ and $q\in S(\A)$ we have
  	\[
  	X\in \ell_{r(p)}^\A(q) \ \Leftrightarrow \ d_qX \in r(p)  \ \Leftrightarrow \ r(p)\in [d_qX]_{\A^d} \ \Leftrightarrow \ p\in [d_qX]_{\B} \ \Leftrightarrow \ d_qX \in p. 
  	\]	
  	So, the map $r$ is injective and hence an homeomorphism.  		
 	\end{proof}

	\section{Translations by external elements}\label{sec:3}
		
	Here we use the set-up established in the previous section to study the Ellis enveloping semigroup of spaces of types, following the approach of \cite[Section 2]{lN12}.

	Fix a $\kappa$-saturated structure $\bar M$ in a language $\LL$ and let $M$ be an elementary substructure with $|M|<\kappa$. Let $G=G(\bar M)\subset \bar M^n$ be an $M$-definable group and write $G(M)$ for $G\cap M^n$. Let $\mathrm{Def}_G(M)$ be the Boolean algebra of all $M$-definable subsets of $G(M)\subset M^n$. We denote by $S_G(M)$ the Stone space of $\mathrm{Def}_G(M)$, that is, the space of ultrafilters of $M$-definable subsets of $G(M)$. Using the correspondence between formulas and definable sets, we can identify an ultrafilter $n$-type over $M$ with an $n$-type ({\it i.e.} a maximal consistent set of $\LL_M$-formulas concentrating on $G$). 
	
	We recall the notion of externally definable subset. 
	
	\begin{definition}\label{def:extdef}
	We say that $X\subset G(M)$ is {\em externally definable} if there exists some $\bar M$-definable set $Y\subset G$  such that $X=Y\cap G(M)$.
	\end{definition}

	 We denote the collection of all externally definable subsets of $G(M)$ by $\Def_G^{\rm ext}(M)$ and we write $S_G^{\rm ext}(M)$ to denote its Stone space. Newelski \cite[Lemma 1.3]{lN12} proved that the Boolean algebra of externally definable subsets of $G(M)$, which is clearly left-invariant under translates of $G(M)$, is $d$-closed.

	Bearing in mind the results from the previous section, we refine the notion of externally definable. We introduce the following concept:
	
	\begin{definition}
We say that a subset $X\subset G(M)$ is {\em $d$-definable} if there is some $M$-definable set $Y\subset G$ and some $h\in G$ such that $X=Yh\cap G(M)$. 
	\end{definition}
It is clear that every $M$-definable subset of $G(M)$ is $d$-definable, and also that every $d$-definable is externally definable. Furthermore, we have the following:

\begin{lemma}\label{L:ddefinable}
	Let $X$ be a subset of $G(M)$. The set $X$ is $d$-definable if and only if $X=d_q(Z\cap M)$ for some $M$-definable subset $Z\subset G$ and some $q\in S_G(M)$. In particular, the Boolean algebra generated by all $d$-definable subsets of $G(M)$ is $\mathrm{Def}_{G}(M)^d$.
\end{lemma}
\begin{proof}
Given an $M$-definable subset $Y\subset G$ and an arbitrary element $h\in G$ we prove that $Yh\cap G(M)=d_q( Y\cap M)$ for  $q\in S_G(M)$ the ultrafilter corresponding to $\tp(h^{-1}/M)$. Indeed, we have for $g\in G(M)$ that
\begin{align*}
 g\in d_q (Y\cap M) \  \Leftrightarrow \ g\inv (Y\cap M)\in q   \ \Leftrightarrow \  h^{-1}\in g\inv Y \ \Leftrightarrow \  g\in Yh \cap G(M).
\end{align*}
This yields that the $d$-definable sets are precisely the sets of the form $d_q(Z\cap M)$ for some $q\in S_G(M)$ and some $M$-definable subset $Z\subset G$. So, we deduce that the Boolean algebra generated by all $d$-definable subsets of $G(M)$ is $\mathrm{Def}_G(M)^d$.
\end{proof}

As a consequence, it follows from Lemma \ref{L:d-closed} that the Boolean algebra generated by the $d$-definable subsets of $G(M)$ is left-invariant and $d$-closed. Hence, combining this with the previous section, we obtain that 
\[
\Lambda : S_G^{\rm ext}(M) \to E(S_G(M)) , \ p\mapsto \ell_p
\]
is always an epimorphism of Ellis semigroups, by Proposition \ref{P:Epi}. Furthermore, we have:
	 
\begin{cor}\label{C:Equiv-Ellis-ext}
The function $\Lambda : S_G^{\rm ext}(M) \to E(S_G(M))$, given by $p\mapsto \ell_p$, is an isomorphism of Ellis semigroups if and only if every externally definable subset of $G(M)$ is a (positive) Boolean combination of $d$-definable sets.
\end{cor}
\begin{proof} Observe first that the complement of a $d$-definable set is again $d$-definable. So, a Boolean combination of $d$-definable sets is indeed a positive Boolean combination. Hence, we obtain the result by Lemma \ref{L:ddefinable} and Corollary \ref{C:Equivalence}.
\end{proof}

	\section{Ellis envelopes in o-minimal structures}\label{sec:4} Fix an $o$-minimal structure $M$ and let $G$ be a definable group. Throughout the section, we shall work within an $|M|^+$-saturated elementary extension $\bar M$ of $M$.

\subsection{One-dimensional groups} In this setting, we state and prove the following: 
	
	\begin{theorem}\label{T:1dim}
		Let $G$ be a definably connected one-dimensional definable group. Then $(S^{\rm ext}_G(M),*)$ and $(E(S_G(M)),\circ)$ are naturally isomorphic as Ellis semigroups.
	\end{theorem}
	\begin{proof}
		Note first that since $G$ is definably connected, it is commutative by \cite[Corollary 2.15]{aP88}. So, along the proof we use additive notation.
		
		By Corollary \ref{C:Equiv-Ellis-ext} it suffices to show that every externally definable subset of $G(M)$ is a Boolean combination of $d$-definable sets. So, fix an $\LL_{\bar M}$-formula  $\psi(x)$ concentrated on $G$. To prove that $\psi(M)$ is a Boolean combination of $d$-definable sets, we will heavily rely on the results from \cite{vR91}. We distinguish two cases:
		
		\noindent {\em Case 1}. Suppose that $G$ is not definably compact; thus $G$ is of $\mathbb R$-type in the terminology from \cite{vR91}. In this case, there is a definable linear order $<_G$ on $G$ such that $(G,+,<_G)$ is an ordered divisible torsion-free abelian group and $<_G$ is dense and without endpoints. Moreover, every definable subset of $G$ is a finite union of $<_G$-intervals and points. Hence, we may assume without loss of generality that $\psi(\bar M)$ is an $<_G$-interval or a point.  Since a point is clearly $d$-definable, we can assume that $\psi(\bar M)$ is an $<_G$-interval of the form $(-\infty,b)=\{x\in G \ | \ x<_G b\}$ or $(a,+\infty)$, which is defined likewise. It then follows that either   
		\[
		\psi(M)=\big((-\infty,0)+b\big)\cap G(M) \ \text{ or } \ \psi(M)=\big((0,+\infty)+a\big)\cap G(M).
		\] 
		In both cases the externally definable set is $d$-definable, as desired.
		
		\noindent {\em Case 2}. Suppose that $G$ is definably compact; thus $G$ is of $\mathbb S^1$-type according to \cite{vR91}. In this case, there exists a definable circular order $R(x,y,z)$ on $G$ such that 
		\begin{enumerate}[i)]
			\item $R(x,y,z)$ implies that $x,y,z$ are distinct,
			\item $R(x,y,z)$ implies that $R(y,z,x)$ and $R(-z,-y,-x)$  hold,
			\item $R(x,y,z)$ implies that $R(x+u,y+u,z+u)$ holds for any $u\in G$.
		\end{enumerate}
		Moreover, for any $x_0\in G$, the relation $R(x_0,y,z)$ defines a dense linear order without endpoints in $G\setminus\{x_0\}$. We write $<_{x_0}$ for this binary relation.  In addition, every definable subset of $G\setminus \{x_0\}$ is a finite union of $<_{x_0}$-intervals and points. In particular, a definably connected definable proper subset of $G$ is of the form
		$R(a,\bar M,b)$ for some $a,b\in G$.
		
		Clearly, we may assume that $\psi(M)\subsetneq G(M)$, as otherwise there is nothing to prove. Thus, after translating by an element of $G(M)$ if necessary, we have that $0\in G(M) \setminus \psi(M)$, so we can assume that $\psi(M)$ is of the form $R(0,M,a)$, with $a\in \bar M \setminus M$. By \cite[Proposition 4]{vR91}, there is only one element of order $2$ in $G$. Denote it by $\frac{1}{2}$ and note that $\frac{1}{2}\in G(M)$, so $a\neq \frac{1}{2}$. We consider two subcases:
		
		\noindent {\em Case 2.1}. If $R(0,a,\frac{1}{2})$, then $R(\frac{1}{2},a+\frac{1}{2},0)$ by iii) adding $u=\frac{1}{2}$, so $R(0,\frac{1}{2},a+\frac{1}{2})$ by ii). Thus, we get $R(0,a,a+\frac{1}{2})$ by transitivity of $<_0$. We claim that
			\[
			R(0,M,a) = R\left(\frac{1}{2}+a, M,a\right) \cap R\left(0,M,\frac{1}{2}\right) ,
			\]
			which yields the result since $\big(R(\frac{1}{2},\bar M,0)+a \big)\cap M = R(\frac{1}{2}+a, M,a)$ is  $d$-definable. Let $x\in G(M)$ be such that $R(0,x,a)$; note then that $R(0,x,\frac{1}{2})$ by transitivity of $<_{0}$ since $R(0,a,\frac{1}{2})$. 
			 On the other hand, as $R(0,a,a+\frac{1}{2})$, we have $R(a,a+\frac{1}{2},0)$ by ii). Combining this with $R(a,0,x)$, we get by transitivity of $<_a$ that $R(a,a+\frac{1}{2},x)$ and hence $R(a+\frac{1}{2},x,a)$, as desired. 
			
			For the right to left inclusion, let $x\in G(M)$ be such that $R(\frac{1}{2}+a, x,a)$ and $R(0,x,\frac{1}{2})$. Note first that $R(0,x,\frac{1}{2})$ and $R(0,\frac{1}{2},a+\frac{1}{2})$ implies $R(0,x,a+\frac{1}{2})$ by transitivity of $<_0$. Thus $R(x,a+\frac{1}{2},0)$ by ii). As $R(a+\frac{1}{2},x,a)$ implies $R(x,a,a+\frac{1}{2})$ by ii), we get that $R(x,a,0)$ by transitivity of $<_x$. This yields that $R(0,x,a)$, as claimed.
		
	\noindent {\em Case 2.2}. If $R(0,\frac{1}{2},a)$, then $R(0,-a,\frac{1}{2})$ by $ii)$, since $\frac{1}{2}$ has order $2$. So, the set $R(0,M,-a)$ is a Boolean combination of $d$-definable sets by Case 2.1. Now, observe using ii) that
			\[
			R(0,x,-a) \text{ holds} \ \Leftrightarrow \ R(a,-x,0) \text{ holds}.
			\]
			Hence, we deduce that $R(a,M,0) = - R(0,M,-a)$, which yields that \[G(M)\setminus R(a,M,0) = R(a,0,M) \cup \{0\}\] is also a Boolean combination of $d$-definable sets. Consequently, so is  $R(0,M,a)$.					
		This finishes the proof. 
	\end{proof}

\subsection{Arbitrary dimension in an $\aleph_0$-saturated model} 	The statement does not generalise to arbitrary dimensions, not even in the semialgebraic case. Our next goal is to prove the following.

\begin{theorem}\label{main} Let $M$ be a $\aleph_0$-saturated o-minimal expansion of a real closed field, and let $G=(M^2,+)$. Then, the continuous function $\Lambda : S_G^{\rm ext}(M) \to E(S_G(M))$ is not an isomorphism of Ellis semigroups. 	
\end{theorem}

In view of Corollary \ref{C:Equiv-Ellis-ext}, to prove the theorem we only need to find an externally definable subset of $M^2$ which is not a Boolean combination of $d$-definable subsets of $M^2$. To this purpose it will be convenient to work within the Shelah expansion of the structure, which we next recall.

Given a definable subset $Z\subset \bar M^n$, defined with parameters from $\bar M$, introduce an $n$-relational symbol $R_{Z}$ associated to $Z$. Define the structure $M^\mathrm{Sh}$ which expands $M$ together with these new relational symbols $R_Z$, where $R_Z$ is interpreted as the externally definable set $Z\cap M^n$. Shelah's expansion theorem \cite{sS09} ensures that the structure $M^\mathrm{Sh}$ has elimination of quantifiers, see also \cite{CS13,mZ10} for a different proof. In particular, given an externally definable set $Y\subset M^{n+1}$ and $x\in M^n$, the fiber of $Y$ over $x$ 
$$
Y_x:=\{y\in M \ | \ (x,y)\in Y\}
$$
is also externally definable, by Shelah's quantifier elimination theorem. In the o-minimal setting, quantifier elimination was first proven by Baisalov and Poizat \cite{BP98}. Furthermore, as a consequence they deduced that $M^\mathrm{Sh}$ is weakly o-minimal \cite[p. 577]{BP98}. In particular:

\begin{fact}\label{F:Convex}
Let $M$ be an o-minimal structure and let $X\subset M$ be an externally definable subset. Then $X$ is a finite disjoint union of convex sets. 
\end{fact}

Our first step towards the proof of Theorem \ref{main} is the following very general statement on a specific family of externally definable sets of $M^{n+1}$, for arbitrary (and non-necessarily) saturated o-minimal structures.

\begin{prop}\label{general}Let $h:D\rightarrow \bar{M}$ be an $\bar{M}$-definable function, where $D\subset \bar{M}^n$ is an $\bar{M}$-definable cell, such that for every $x\in D\cap M^n$ the type $\tp(h(x)/M)$ is not definable. Let 
	\[
	X=\{(x,y)\in M^{n+1} \ | \ x\in D \text{ and } y<h(x)\}
	\]
	and let $\{C_{ij}\}$ be a finite family of externally definable subsets of $M^{n+1}$ such that for $Y_i:=\bigcap^r_{j=1} C_{ij}$ we have that
	$X=\bigcup_{i=1}^\ell Y_i$.
	Then there is a family of externally definable subsets $\{F_{ij}\}_{i,j}$ of $D\cap M^n$ with $D\cap M^n=\bigcup_{i,j} F_{ij}$,  and for each $x\in F_{ij}$ there are $y_1,y_2\in M$ with $y_1<h(x)<y_2$ and $(y_1,y_2)\cap C_{ij,x}=(y_1,h(x))\cap M$.
\end{prop}
\begin{proof}First note that since each type $\text{tp}(h(x)/M)$ is not definable, for each $x\in M$ there is some $\varepsilon\in M_{>0}$ such that $-\varepsilon <h(x)<\varepsilon$. Moreover, if $y\in M$ is such that $y<h(x)$ then the set $(y,h(x))\cap M$ is infinite. Indeed, otherwise the type $\tp(h(x)/M)$ would be defined by the formula $y> b$ for some $b\in M$ with $b<h(x)$ and $(b,h(x))\cap M=\emptyset$. Similarly,  if $y\in M$ is such that $y>h(x)$ then the set $(h(x),y)\cap M$ is infinite.
	
Now,  for each $i=1,\ldots,\ell$, consider the set
\[
\tilde Y_i = \left\{ (x,y)\in X \ | \ (y,+\infty) \cap Y_{i,x} = (y,h(x))\cap M  \right\}
\] 
and let $\pi(\tilde Y_i)$ denote its projection on the first $n$ coordinates, that is,
\[
\pi(\tilde Y_i):=\{x\in D\cap M^n \ | \  \exists y_0\in M \text{ with } (x,y_0)\in \tilde Y_i \}.
\]
Both sets are clearly definable in the structure $M^\mathrm{Sh}$ and therefore they are externally definable.

	\medskip
	\noindent{\emph{Claim 1}.} If $x\in D\cap M^n$ is not in $\pi(\tilde Y_i)$, then there is an element $y_0\in M$ with $y_0<h(x)$, {\it i.e.} $(x,y_0)\in X$, such that  $(y_0,h(x))\cap Y_{i,x}=\emptyset$. 
	
	\smallskip
	\noindent{\emph{Proof of Claim 1.}} Assume otherwise that $(y_0,h(x))\cap Y_{i,x}\neq \emptyset$ whenever $y_0\in M$ and $y_0<h(x)$. Note that since the set $Y_{i,x}$ is externally definable, it is a finite union of maximal (disjoint) convex sets $Z_1<\ldots<Z_s$ by Fact \ref{F:Convex}. Choose some $y'\in Z_s$. We next show that 
	$$
	(y',+\infty)\cap Y_{i,x}=(y',h(x))\cap M,
	$$
	which will give the desired contradiction since $x\notin \pi(\tilde Y_i)$. To prove this equality, note first that 
	\[
	(y',+\infty)\cap Y_{i,x}=(y',h(x))\cap Y_{i,x},
	\] as $Y_i\subset X$. So, it suffices to see the inclusion $\supset$. Let $y_0\in (y',h(x))\cap M$. We have by assumption that $(y_0,h(x))\cap Y_{i,x}\neq \emptyset$, so there is some element $y_1\in (y_0,h(x))\cap Y_{i,x}$. 
	Now, recall that 
	\[
	Y_{i,x} = Z_1\cup \ldots\cup Z_s \text{ with } Z_1<\ldots<Z_s.
	\]
	Also $y'\in Z_s$ and $y'<y_0<y_1$ with $y_1\in Y_{i,x}$. Thus, we deduce that $y_1\in Z_s$ and therefore by convexity $y_0\in Z_s\subset Y_{i,x}$, as required. \hfill $\square${\tiny Claim 1}
	\medskip
	
	\medskip
	\noindent{\emph{Claim 2}.} We have $D\cap M^n=\bigcup_{i=1}^\ell \pi(\tilde Y_i)$. 
	
	\smallskip
	\noindent{\emph{Proof of Claim 2.}} Let $x_0\in D\cap M^n$ and suppose that $x_0\notin \pi(\tilde Y_i)$ for every $i=1,\ldots,\ell$. By Claim 1 there is some element $y_0\in M$ with $(x_0,y_0)\in X$ such that $(y_0,h(x_0))\cap Y_{i,x_0}=\emptyset$ for every $i=1,\ldots,\ell$. On the other hand, since $X=\bigcup_{i=1}^\ell Y_i$, we also have 
	\[
	(-\infty,h(x_0)) \cap M = \bigcup_{i=1}^\ell Y_{i,x_0}.
	\] 
	However, since the set $(y_0,h(x_0))\cap M$ is infinite, as remarked at the very beginning of the proof, we obtain a contradiction. 
	\hfill $\square${\tiny Claim 2}
	\medskip

	
	Now, recall that each $Y_i= \bigcap_{j=1}^r C_{ij}$ where each $C_{ij}\subset M^{n+1}$ is a externally definable subset. Given $i=1,\ldots,\ell$ and  $j=1,\ldots,r$, define the set
	\[
	F_{ij}=\{x\in \pi(\tilde Y_i) \ | \ \exists y_1,y_2\in M,\,  y_1<h(x)<y_2 \text{ and } (y_1,y_2) \cap C_{ij,x} = (y_1,h(x))\cap M\},
	\]
	which is externally definable.
	
	To finish the proof, it is enough to prove that 
	\[
	D\cap M^n=\bigcup_{i=1}^\ell \bigcup_{j=1}^r F_{ij}.
	\]	
	Let $x\in D\cap M^n$. By Claim 2 there is some $i=1,\ldots,\ell$ with $x\in \pi(\tilde Y_i)$. Thus,   there is an element $y_1\in M$ with $y_1<h(x)$ such that $(y_1,+\infty)\cap Y_{i,x}=(y_1,h(x))\cap M$. Since $Y_i=\bigcap_{j=1}^r C_{ij}$, for all $j=1,\ldots,r$ we have $(y_1,h(x))\cap M \subset C_{ij,x}$ and hence 
	$$
	(y_1,h(x))\cap M = (y_1,h(x))\cap C_{ij,x},
	$$
	since $C_{ij,x}\subset M$. 	
	Therefore, we need only prove that there exist some $j=1,\ldots,r$ and  some $y_2\in M$ with $h(x)<y_2$ for which $(y_1,h(x))\cap C_{ij,x} = (y_1,y_2)\cap C_{ij,x}$. 
	
	Recall that if $\xi\in M$ is such that $\xi>h(x)$ then $(h(x),\xi)\cap M$ is infinite. Furthermore, we first claim: 
	
	\medskip
	\noindent{\emph{Claim 3}.} There exists some $j\in \{1,\ldots,r\}$ such that for every $\xi\in M$ with $\xi>h(x)$ we have $(h(x),\xi)\cap M\not\subset C_{ij,x}.$
	
	\smallskip
	\noindent{\emph{Proof of Claim 3.}} Otherwise, for each $j$ there is some $\xi_j\in M$ with $\xi_j>h(x)$ such that $(h(x),\xi_j)\cap M\subset C_{ij,x}$. Set  $\xi':=\min\{\xi_1,\ldots,\xi_r\}$ and note then that 
	\[
	(h(x),\xi')\cap M\subset (y_1,+\infty)\cap Y_{i,x}=(y_1,h(x))\cap M,
	\]
	where the last equality holds by the choice of $y_1$. However, the above is clearly a contradiction.
	\hfill $\square${\tiny Claim 3}
	\medskip
	
	Let $j$ be given by Claim 3 and consider the non-empty externally definable set
	\[
	\{y'\in M \ | \ y'>h(x) \text{ and } y'\in M\setminus C_{ij,x}\}.
	\]
	By Fact \ref{F:Convex} this is a finite union of maximal (disjoint) convex sets $Z_1< \ldots<Z_s$. Fix an arbitrary element $y_2\in Z_1$. We claim that $(h(x),y_2)\cap M\subset Z_1$. Indeed, given and element $\xi\in (h(x),y_2)\cap M$, there is some  $\xi'\in (h(x),\xi)\cap M$ with $\xi'\in M\setminus C_{ij,x}$, by the choice of $j$. So, we have that $\xi'\in Z_k$ for a certain $k=1,\ldots,s$. Since $\xi'<\xi<y_2$ we must have that $\xi'\in Z_1$ and by convexity $\xi\in Z_1$, as desired.
	
	In particular, once we know that $(h(x),y_2)\cap M\subset Z_1$, it follows that 
	$$
	(h(x),y_2)\cap M \subset M\setminus C_{ij,x}
	$$
	and so 
	\[
	(y_1,y_2)\cap C_{ij,x} = \big((y_1,h(x)) \cap C_{ij,x} \big) \cup \big((h(x),y_2) \cap C_{ij,x} \big) = (y_1,h(x)) \cap C_{ij,x}.
	\]
	 This finishes the proof.\end{proof}	

Now, we introduce some notation. Let $M$ denote an o-minimal expansion of a real closed field. Denote by $\inf^+_M(\mathbb R)$ the set of positive elements of $M$ which are infinitesimal with respect to $\mathbb R$, that is
\[
\inf\nolimits^+_M(\mathbb{R}):=\left\{x\in M_{>0} \ | \ x<\tfrac{1}{n} \text{ for all }n\in \N \setminus\{0\} \right\}.
\]
Consider the unique type $p_{\text{id}}(x)\in S_1(M)$ determined by the set of formulas
\[
\left\{x<\tfrac{1}{n} \ | \ n\in \mathbb{N}\setminus\{0\} \right\} \cup \left\{m<x \ | \ m\in \inf\nolimits^+_M(\mathbb{R})\right\}.
\]
The subscript $\mathrm{id}$ stands for idempotent, since the type $p_{\text{id}}$ is idempotent in the space of types with respect to the semigroup operation induced by addition.

Now, Theorem \ref{main} is a straightforward consequence of Corollary \ref{C:Equiv-Ellis-ext} and the following proposition.

\begin{prop}\label{ex}Let $M$ be an $\aleph_0$-saturated o-minimal expansion of a real closed field. Let $\bar M$ be an $|M|^+$-saturated extension of $M$ and let $e\in \bar M$ be a realization of $p_{\rm id}$. The externally definable subset 
	$$
	X=\{(x,y)\in M^2 \ | \ 0<x \text{ and } 0<y<e\cdot x\}
	$$
	of the additive group $(M^2,+)$ is not a Boolean combination of $d$-definable sets.
\end{prop}

\begin{remark}The underlying intuition behind the above statement is clear: the multiplication map $x\mapsto e\cdot x$ cannot be equal to a  map of the form 
	$$g(x):=f(x+c)+d,$$
	where $f$ is some $M$-definable map and $c,d\in\bar{M}$. In fact, the latter is clear: if $e\cdot x=g(x)=f(x+c)+d$ for all $x\in \bar M_{>0}$ then $f(x)=e(x-c)-d$ and since $f$ is $M$-definable then it is easy to deduce that $e\in M$, which is a contradiction. However, our situation is completely different: what we really have to show is that the map $x\mapsto \tp(e\cdot x/M)$ cannot be equal to a map $x\mapsto \tp(g(x)/M)$ where $g(x)$ is as above. If such is the case then \emph{essentially} $\tp(g'(x)/M)=\tp(e/M)$ and  \emph{therefore} $g(x)=e\cdot x+k$ for some $k\in \bar M$, from where we obtain easily another contradiction. The two words in emphasis are far from clear, see Example \ref{extddef} below. 
\end{remark}

In the proof of Proposition \ref{ex} we try to make sense of the above intuition through a case-by-case development via the mean value theorem (see \cite[Ch.\,7]{vdD98}).

\begin{proof}[Proof of Proposition \ref{ex}] In an o-minimal structure, each definable set over a set of parameters is a finite disjoint union of cells defined over the same set of parameters. Thus, a Boolean combination of $d$-definable sets of $(M^2,+)$ is a Boolean combination of $d$-definable cells, {\it i.e.}, sets of the form $C=\big (\widetilde{C}+(c,d)\big) \cap M^2$ where $\widetilde{C}$ is an $M$-definable cell and $(c,d)\in \bar M^2$. 

Suppose that the given set $X$ is a Boolean combination of $d$-definable cells. We can write it as  $X=\bigcup_i Y_i ,$ where  $Y_i:=\bigcap^{r}_{j=1} C_{ij}$ for some $d$-definable cells $C_{ij}$. We first see that we may reduce to the case when $X$ is a specific single $d$-definable cell, after possibly shrinking the set $X$. We prove:

\noindent \emph{Claim 1.} There are some $a\in M$ with $a>0$, an $M$-definable function $f:I\to \bar M$ with $I$ an open interval, and elements $c,d\in\bar M$ such that $M_{>a}\subset I-c$ and
\begin{multline}\label{ME}
\left\{(x,y)\in M^2 \ | \ (a<x ) \wedge (0<y<e\cdot x) \right\}=  \\
=\left\{(x,y)\in M^2 \ | \ (a<x )\wedge (0<y<f(x+c)+d) \right\}. \tag{ME}
\end{multline}

\noindent\emph{Proof of Claim 1.} Note that $\text{tp}(e/M)$ is not definable, and therefore for every $x\in M_{>0}$ we have that $\text{tp}(e\cdot x/M)$ is neither definable. Hence, by Proposition \ref{general}, there is a family of externally definable subsets $\{F_{ij}\}$ of $M_{>0}$ such that $M_{>0}=\bigcup F_{ij}$,  and for each $x\in F_{ij}$ there are $y_1,y_2\in M$ with $y_1<e\cdot x<y_2$ and $(y_1,y_2)\cap C_{ij,x}=(y_1,e\cdot x)\cap M$. Note that we can choose $y_1$ with $y_1>0$. 

Each $F_{ij}$ is a finite union of convex sets, by Fact \ref{F:Convex}. Therefore there exist a pair $(i,j)$ and an element $a\in M_{>0}$ such that $M_{>a}\subset F_{ij}$. Denote $C:=C_{ij}$. So, for all $x\in M_{>a}$ there exist some $y_1\in (0,e\cdot x)\cap M$ and some $y_2\in (e\cdot x,+\infty)\cap M$ with
$$
(y_1,y_2) \cap C_x=(y_1,e\cdot x)\cap M.
$$ 
By definition, the $d$-definable cell $C$ equals to  $(\widetilde{C}+(-c,d)\big) \cap M^2$ for some element $(-c,d)\in \bar M^2$ and some $2$-dimensional cell 
$$
\widetilde{C}:=(f_1,f_2)= \left\{(x,y)\in I\times \bar M \ | \ f_1(x)<y<f_2(x)\right\},
$$
where $I\subset \bar M$ is an open interval, defined over $M$, and $f_1,f_2:I\rightarrow \bar M$ are two $M$-definable functions. Note that
\begin{align*}
\big(\widetilde{C}+(-c,d)\big) & =\left\{(x-c,y+d)\in (I-c)\times \bar M \ | \ f_1(x)<y<f_2(x)\right\} \\
& =\left\{(x,y)\in (I-c)\times \bar M \ | \ f_1(x+c)+d<y<f_2(x+c)+d \right\} .
\end{align*}
Thus, since $C=(\widetilde{C}+(-c,d)\big) \cap M^2$, for each $x\in M_{>a}$ there exist $y_1\in (0,e\cdot x)\cap M$ and $y_2\in (e\cdot x,+\infty)\cap M$ with
$$
(y_1,e\cdot x)\cap M= (y_1,y_2)\cap C_x=(y_1,y_2)\cap \big(f_1(x+c)+d,f_2(x+c)+d\big)\cap M.
$$
In particular, it follows that $y_1<f_2(x+c)+d$ and $f_1(x+c)+d \le e\cdot x<y_2$. So, taking the union with the interval $(0,y_1]\cap M$ in the above equalities, we get
\[
(0,e\cdot x)\cap M= (0,y_2)\cap \big(0,f_2(x+c)+d\big)\cap M.
\]
Note we must have that $f_2(x+c)+d <y_2$ as otherwise $(0,e\cdot x)\cap M= (0,y_2)\cap M$, which is a contradiction. Altogether, it follows readily that
\begin{multline*}
\{(x,y)\in M^2 \ | \  (a<x )\wedge (0<y<e\cdot x)\}=  \\
=\{(x,y)\in M^2 \ | \ (a<x )\wedge (0<y<f_2(x+c)+d)\},
\end{multline*}
as desired. \hfill $\square${\tiny Claim 1}
\medskip

Hence, to finish the proof it is enough to contradict Claim 1. In other words, we need to show that there cannot exist $a\in M_{>0}$ and $c,d\in \bar M$ and an $M$-definable function $f:I\rightarrow \bar M$ satisfying (\ref{ME}). The term ME stands for \emph{main equality}. 

Let $a\in M$, $c,d\in \bar M$ and $f:I\to \bar M$ be given by Claim 1. Note that we are assuming that we can evaluate $f$ in $x+c$ for each $x\in M_{>a}$, since we have $M_{>a}\subset I-c$.

\noindent \emph{Claim 2.} We can assume that $I\subset \bar M$ is an open interval definable over $M$ satisfying that:
\begin{itemize}
	\item the function $f$ is $\mathcal{C}^1$,
	\item either $f>0$ or $f<0$ holds, and
	\item it holds $f'>0$ and in particular $f$ is strictly increasing.
\end{itemize}

\noindent\emph{Proof of Claim 2.} Indeed, by o-minimality there are open disjoint subintervals $I_1,\ldots,I_s$ of $I$, all definable over $M$, with $F:=I\setminus (I_1\cup \cdots \cup I_s)$ finite and such that each $f|_{I_i}$ is a $\mathcal{C}^1$-function satisfying that either $f|_{I_i}>0$ or $f|_{I_i}<0$ or $f|_{I_i}=0$, and either $f'|_{I_i}>0$ or $f'|_{I_i}<0$ or $f'|_{I_i}=0$. 

Each $J_i:=(I_i-c)\cap M_{>a}$ is a convex subset of $M_{>a}$ with $J_i\cap J_j=\emptyset$ if $i\neq j$. Since $J_i$ is convex, either there is an element $b\in M_{>a}$ such that $J_i\subset M_{<b}$, in which case we will say that $J_i$ is {\em bounded}, or there is some $b\in M_{>a}$ such that $M_{>b}\subset J_i$. Thus, as $M_{>a}\subset I-c$, we have that  
\[
M_{>a}\subset (I-c)=(F-c) \cup \bigcup_{i=1}^s J_i
\]
 and therefore there is $a_1\in M_{>a}$ such that $M_{>a_1}\subset \bigcup_{i=1}^s J_i$. In particular, not all $J_i$ can be bounded. So, there are an $i=1,\ldots,s$ and some  $a_2\in M_{>a_1}$ such that $M_{>a_2}\subset J_i$. In particular, in (\ref{ME}) we can replace $a$ by $a_2$ and $f$ by $f|_{I_i}$, as required as we will see next. 

Note that in (\ref{ME}), the case $f=0$ is not allowed. Moreover, the case $f'\le 0$ is also excluded. To see the latter, consider the $\bar M$-definable function $g(x):=f(x+c)+d$, which is differentiable over its domain $I-c$. By the mean value theorem, for all $x_1,x_2\in M_{>a}$ with $x_1<x_2$ there is some  $u\in \bar M$ with $x_1<u<x_2$ such that
\[
g(x_2)-g(x_1) = g'(u)\cdot (x_2-x_1)=f'(u+c)\cdot (x_2-x_1).
\] 
If $f' \le 0$, then we have that $g(x_2)\le g(x_1)$ for all $x_1,x_2\in M_{>a}$ with $x_1<x_2$. However, for a fixed $x_1\in M_{>a}$, choose some $m_1,m_2\in M$ with $0<m_1<m_2<e$. Let $x_2=m_1^{-1} \cdot x_1$ and $y=m_2\cdot x_2$. Thus, we have that $ex_1<x_1< y < e\cdot x_2$. So, we get by (\ref{ME}) that $y<g(x_2)$ but $x_1<y\not < g(x_1)$, a contradiction.
\hfill $\square${\tiny Claim 2}
\medskip

Henceforth, we define 
$$g(x):=f(x+c)+d$$
which is a differentiable over its domain $I-c$ and it is definable over $\bar M$. Moreover, note that for all $x\in M_{>a}$ we have that $0<g(x)$, by (\ref{ME}). 

From now on, we will use the mean value theorem without explicit reference.

\medskip
\noindent \emph{Claim 3.} We can assume that the interval $I$ is of the form $(i,+\infty)$ for some $i\in M$ or of the form $(-\infty,i)$ for some $i\in M$. Moreover, in the former case we have that $c>i-(a+1)$ and in latter case we have $c<M$.	

\smallskip

\noindent\emph{Proof of Claim 3.} Note first that $c<M$ whenever $I\subset (-\infty,i)$ for some $i\in M$, since in this case $M_{>a} \subset I-c$ and hence $M<i-c$. Also, if $I\subset (i,+\infty)$ for some $i\in M$, then $c>i-(a+1)$. After observing this, note also that the interval $I$ cannot be of the form $(i_1,i_2)$ for some $i_1,i_2\in M$ with $i_1<i_2$, as otherwise we would have $i_1-(a+1)<c <M$, a contradiction.

Hence, it remains to consider the case when $I=(-\infty,+\infty)$. In that case, if there is some $c_1\in M$ such that $c_1<c$ then for every $x\in M_{>a}$ we have that $a+c_1<x+c$. In which case, we set $I'=(a+c_1,+\infty)$. Otherwise, we would have that $c<M$ and set $I'=(-\infty,0)$. In either case we have that $I'$ satisfies that $M_{>a}\subset I'-c$. Hence, we can replace $I$ by $I'$, obtaining the claim.
	\hfill $\square${\tiny Claim 3}
\medskip

From now on we assume that $I$ is of the form given in Claim 3. To get the desired contradiction, now we study how the derivative of $f$ behaves on $I$.

\medskip
\noindent \emph{Claim 4.} There is some $M$-definable unbounded open interval $J\subset I$ such that $f'(x)<e$ for every $x\in J\cap M$. In particular, for every natural number $n\ge 1$ and every $x\in J$ we have that $f'(x)<\tfrac{1}{n}$.

\smallskip
\noindent\emph{Proof of Claim 4.} Since $f'$ is an $M$-definable function on $I\subset \bar M$, we have that
\[
I\cap M=\left\{ x \in I \cap M \ | \ f'(x)<e \right\}  \sqcup \left\{ x \in I\cap M \ | \ f'(x) > e \right\}.
\]
The two disjoint sets from the right hand side of the equality are externally definable. So, both equal a finite union of convex sets by Fact \ref{F:Convex}. This yields that only one of them can be unbounded. In particular, there exists some unbounded open $M$-definable interval $J\subset I$ such that $f'|_{J\cap M} <e$ or $f'|_{J\cap M} >e$.

Suppose, to get a contradiction, that $f'(x)>e$ for every $x\in J\cap M$. As $f'$  and $J$ are both definable over $M$, it follows by $\aleph_0$-saturation of $M$ that there is some natural number $n_0\ge 1$ such that $f'(x)>\tfrac{1}{n_0}$ for every $x\in J$, since clearly
\[
J\cap M = \bigcup_{n\ge 1} \left\{ x\in J \cap M \ | \ f'(x) >\tfrac{1}{n}\right\}.
\]
Now, choose some  $x_1,x_2\in M_{>a}$ with $x_2=2x_1$. Furthermore, by Claim 3 we can choose them so that $[x_1+c,x_2+c]\subset J$. Indeed, if $I$ is of the form $(-\infty,i)$ then by Claim 3 we know that $c<M$ and so both $x_i +c\in J$, since $J$ is of the form $(-\infty,j)$ for some $j\in M$ with $j<i$. On the contrary, if $I$ is of the form $(i,+\infty)$ with $i\in M$, then $J=(j,+\infty)$ for some $j\in M_{>i}$ and also $c \not < M$ by Claim 3. So, if $M<c$ then choose any $x_1\in M_{>a}$. Otherwise, it suffices to choose $x_1 > a+|j - c |$ to ensure that $x_1,x_2\in M_{>a}$ and that $x_1+c$ and so $x_2+c$ belong to $J$. 

As $x_1<x_2$, there is an element $\hat x\in\bar M$ with $x_1<\hat x<x_2$ such that 
\[
g(x_2) - g(x_1) = g'(\hat x)\cdot (x_2-x_1) = f'(\hat x + c)\cdot x_1 > \frac{x_1}{n_0},
\]
where the inequality holds since  $\hat x+c\in J$. Thus $g(2x_1)>g(x_1) + \tfrac{x_1}{n_0}$. Now, fix some $ m\in M$ such that $0<m<e$. Since $m\cdot x_1<e\cdot x_1$, we have by (\ref{ME}) that $m\cdot x_1 < g(x_1)$ and therefore
\[
m\cdot x_1 + \frac{1}{n_0} \cdot x_1 < g(x_1) + \frac{1}{n_0} \cdot x_1  < g(2x_1).
\]
So, again by (\ref{ME}) we obtain that  $m\cdot x_1 + \tfrac{1}{n_0} \cdot x_1 < e\cdot 2x_1$ and hence $\tfrac{1}{2}(m+\tfrac{1}{n_0}) <e$, which is a contradiction. This yields the claim. 
\hfill $\square${\tiny Claim 4}
\medskip

Now, let $J\subset I$ be the open interval given by Claim 4. This is unbounded and definable over $M$. Since the function $f'$ is definable over $M$, by Claim 2 and 4 the limit $\ell=\lim_{|x|\to +\infty} f'(x)$ belongs to $M$ and satisfies that $0<\ell <e$.  Thus, for every $\varepsilon\in M$ with $0<\varepsilon$ there is some unbounded open interval $J_\varepsilon\subset J$, definable over $M$, such that $|f'(x) - \ell|<\varepsilon$ for every $x\in J_\varepsilon$. 

Fix $\varepsilon=\ell$. Let $x_1\in M_{>a}$ be such that $x_1 + c\in J_\ell$; this is possible arguing as in the proof of Claim 4. Let now $m\in M$ be with $0<m<e$ and set $x_2=x_1\cdot m^{-1}$, which is also an element of $M_{>a}$ with $x_1<x_2$. It follows as in the proof of Claim 3 that $x_2+c\in J_\ell$. Hence, there is some $\hat x\in \bar M$ with $x_1<\hat x <x_2$, and so $\hat x+c\in J_\ell$, such that 
\[
g(x_2)-g(x_1)=g'(\hat{x})\cdot (x_2-x_1) = f'(\hat{x}+c)\cdot (x_2-x_1)< 2\ell \cdot  (x_2-x_1).
\]
On the other hand, for $y=x_1+2\ell\cdot (x_2-x_1)\in M$ we have 
\[
0<y < x_1+ 2 \ell \cdot x_2=m \cdot x_2+ 2\ell\cdot  x_2=(m+2\ell)\cdot x_2<e\cdot x_2
\]
and therefore by (\ref{ME}) it follows that $y<g(x_2)$. Thus, we obtain
$$
x_1+2\ell\cdot (x_2-x_1) = y < g(x_2) < g(x_1)+2\ell\cdot (x_2-x_1),$$
which yields that $x_1<g(x_1)$. So, again  by (\ref{ME}) we have that $x_1<e\cdot x_1$, a contradiction as $e<1$. This final contradiction finishes the proof.
\end{proof}	

Finally, let us point out that there are externally definable sets that may not seem a Boolean combination of $d$-definable,  but actually are.

\begin{example}\label{extddef}Let $M$ be an $\aleph_0$-saturated o-minimal expansion of a real closed field. Let $\bar M$ be an $|M|^+$-saturated extension of $M$ and let $e\in \bar M$ be a realization of $p_{\text{id}}$. We show that the externally definable set
	\begin{multline*}X:=\{(x,y)\in M^2 \ | \ (0<x<e)  \wedge \, (0<y< e\cdot x)\}=\\
		=\{(x,y)\in M^2 \ | \ x\in \text{inf}^+_\mathbb{R}(M) \wedge \, (0<y< e\cdot x)\}\end{multline*}
	is an intersection of two $d$-definable sets.
	
	For each $a\in M_{>0}$ denote by $p^{(a)}_{\rm id}:=\tp(e\cdot a/M)$, which is the unique type that extends the set of formulas
	$$
	\left\{x<\tfrac{a}{n} \ | \ n\in \mathbb{N}\setminus\{0\}\right\}\cup \left\{a\cdot m <x \ | \  m\in \text{inf}^+_\mathbb{R}(M)\right\}.
	$$
	Note that for $c\in \text{inf}^+_{\mathbb{R}}(M)$ and $d\in \bar M$ realizing $ p^{(c)}_{\text{id}}$, we have 
	\begin{enumerate}
	\item[$(\ast)$] the element	$\frac{c^2}{2}+d$ realizes $p^{(c)}_{\text{id}}$.
	\end{enumerate}
	Indeed, if there is some infinitesimal $m\in \text{inf}^+_{\mathbb{R}}(M)$ with $\frac{c^2}{2}+d < m\cdot c$ then
	$d<c\cdot (m-\tfrac{c}{2})$. Since $c,d>0$, we must have that $m-\tfrac{c}{2}>0$ and so  $m-\tfrac{c}{2}\in \text{inf}^+_{\mathbb{R}}(M)$. However, this contradicts the fact that $d$ realizes $p^{(c)}_{\text{id}}$.
	
	On the other hand, assume that there is some natural number $n\ge 1$ such that $\frac{c^2}{2}+d \ge\frac{c}{n}$. Since $c\in \text{inf}^+_{\mathbb{R}}(M)$ we have $n\cdot c<1$ and therefore 
	$$
	d\ge \left(\frac{1}{n}-\frac{c}{2} \right)\cdot c =\frac{2-n\cdot c}{2n}\cdot c >\frac{1}{2n}\cdot c,
	$$
	which yields that $d$ does not realize $p^{(c)}_{\text{id}}$, a contradiction.

	Finally, once we have seen $(\ast)$, consider the $\emptyset$-definable sets 
	$$
	Z_1:=\{(x,y)\in \bar M^2\ | \ y\leq \tfrac{1}{2} \cdot x^2\} \quad \text{ and }\quad Z_2:=(-\infty,0)\times \bar M,
	$$
	as well as their associated $d$-definable sets
	$$
	X_1:=M^2\cap \left(Z_1-(e,\tfrac{e^2}{2})\right) = \left\{(x,y)\in M^2 \ | \ y\leq \tfrac{1}{2}\cdot (x+e)^2-\tfrac{1}{2}\cdot e^2 =\tfrac{x^2}{2} +e\cdot x\right \}
	$$
	and
	$$X_2:=M^2\cap \big(Z_2+(e,0)\big)=\left\{(x,y)\in M^2 \ | \ x<e \right\}.$$
	In particular,  the externally definable set
	$$X_1\cap X_2 \cap \big((0,\infty)\times (0,\infty)\big)=\left\{(x,y)\in M^2 \ | \ (0<x<e) \wedge \big(0< y\leq \tfrac{1}{2}\cdot x^2+x\cdot e\big)\right\}$$
	is an intersection of two $d$-definable sets and by $(\ast)$ is clearly equal to $X$. Indeed, given a pair $(x,y)\in M^2$ with $0<x<e$, we see that $0<y<e\cdot x$ if and only if $y\in M_{>0}$ is strictly smaller than any realization of $p^{(x)}_{\text{id}}$. So, the condition $0<y<e\cdot x$ is equivalent to $0< y\leq \tfrac{1}{2}x^2+x\cdot e$, by $(\ast)$, as required.
\end{example}

\subsection{Arbitrary dimension in the real algebraic numbers} Given a real closed field extension $\bar R$ of $\mathbb{R}$ denote by $\mathrm{Fin}_\mathbb{R}(\bar R)$ the set of finite non-standard reals, {\it i.e.}
\[
\mathrm{Fin}_\mathbb{R}(\bar R) = \left\{ x\in \bar R \ | \ \text{ there is some $r\in\mathbb R$ with } |x|<r   \right\}.
\]
Recall that the real field is Dedekind complete. Therefore, for any real closed field extension $\bar R$ of $\mathbb{R}$ there exists a well-defined \emph{standard map} $$\mathrm{st}:\mathrm{Fin}_\mathbb{R}(\bar R)\rightarrow \mathbb{R}$$
which sends an element $x\in \mathrm{Fin}_\mathbb{R}(\bar R)$ to the unique element $\mathrm{st}(x)\in \mathbb{R}$ such that $x\in \mathrm{st}(x)+\mathrm{inf}_\mathbb{R}(\bar R)$.  More generally, given a natural number $n\in \mathbb{N}$ with $n\ge 1$ we define 
\[
\mathrm{st}:\mathrm{Fin}^n_\mathbb{R}(\bar R)\rightarrow \mathbb{R}^n, \ \bar x=(x_1,\ldots,x_n)\mapsto \mathrm{st}(\bar x):=(\mathrm{st}(x_1),\ldots,\mathrm{st}(x_n)).
\]
 Note that given  $\bar x\in  \mathrm{Fin}^n_\mathbb{R}(\bar R)$, the tuple $\mathrm{st}(\bar x)\in \mathbb{R}^n$ is the unique one in $\mathbb{R}^n$ such that $\bar x$ belongs to all its semialgebraic open neighborhoods defined over $\mathbb{R}$.

\begin{prop}\label{prop:realalg}Let $\mathbb{R}_{\rm alg}$ denote the field of real algebraic numbers with the field structure. Let $I\subset \mathbb{R}$ be an open interval and let $h:I\rightarrow \mathbb{R}$ be a semialgebraic function. If the externally semialgebraic subset
$$
X=\{(x,y)\in (I\cap \mathbb{R}_{\rm alg})\times \mathbb{R}_{\rm alg} \ | \ y<h(x)\} 
$$   
of the semialgebraic group $(\mathbb{R}^2_{\rm alg},+)$ is a Boolean combination of $d$-definable sets, then there are finitely many disjoint open subintervals $I_1,\ldots, I_s\subset I$ such that $I\setminus (I_1\cup \cdots \cup I_s)$ is finite and for each $i=1,\ldots,s$ there are a semialgebraic map $f_i:J_i\rightarrow \mathbb{R}$ defined over $\mathbb{R}_{\mathrm{alg}}$, with $J_i$ an open interval, and elements $c_i,d_i\in \mathbb{R}$ such that each $I_i\subset J_i-c_i$ and $h(x)=f_i(x+c_i)+d_i$ for all $x\in I_i$.
\end{prop}

\begin{proof} Note first that if $h(x)=r$  for some $r\in \mathbb R$, then the statement is clear by setting $s=1$, $I_1=I$, $f_1=0$, $c_1=0$ and $d_1=r$. Therefore, we may assume that $h:I\to \mathbb R$ is not a constant function. So, by o-minimality we can assume that $h:I\to \mathbb R$ is strictly monotone and $\mathcal{C}^1$.  By Proposition \ref{general} and similarly as in the proof of Claim 1 of Proposition \ref{ex}, we can further assume that there is a saturated real closed field extension $\bar R$ of $\mathbb{R}$ and a semialgebraic map $f:J\rightarrow \bar R$ definable over $\mathbb{R}_{\rm alg}$ with $J\subset \bar R$ an open interval and $c,d\in \bar R$ such that $I\cap \mathbb{R}_{\rm alg} \subset J-c$ and
$$
X=\{(x,y)\in (I\cap \mathbb{R}_{\rm alg})\times \mathbb{R}_{\rm alg} \ | \ y<f(x+c)+d\}.
$$
Again by o-minimality we can assume that  $f$ is strictly monotone and $\mathcal{C}^1$, as in Claim 2 of Proposition \ref{ex}.
Note that $I\cap \mathbb{R}_{\rm alg}\subset J-c$ implies $I\subset J-c$, by density.

Moreover, since 
\[
h(x)\equiv _{\mathbb{R}_{\rm alg}} f(x+c)+d
\] 
for all $x\in I\cap \mathbb{R}_{\rm alg}$, we have that $f(x+c)+d\in \mathrm{Fin}_\mathbb{R}(\bar R)$ and so $\text{st}(f(x+c)+d)=h(x)$ for all $x\in I\cap \mathbb{R}_{\rm alg}$. Note by monotonicity of $f$ that we also have $f(x+c)+d\in \mathrm{Fin}_\mathbb{R}(\bar R)$ for all $x\in I\subset \mathbb R$.

Now, we distinguish three cases.

\setcounter{substep}{0}

\begin{substeps} Assume $c\in \mathrm{Fin}_\mathbb{R}(\bar R)$, so clearly also $d\in \mathrm{Fin}_\mathbb{R}(\bar R)$. Set $c':=\mathrm{st}(c)$ and $d':=\mathrm{st}(d)$. We first note that since $J$ is defined over $\mathbb{R}_{\rm alg}$ we have
$$I\cap \mathbb{R}_{\rm alg}=\mathrm{st}(I\cap \mathbb{R}_{\rm alg})\subset \mathrm{st}((J\cap \mathrm{Fin}_\mathbb{R}(\bar R))-c )=\mathrm{st}(J\cap \mathrm{Fin}_\mathbb{R}(\bar R))-c'\subset \mathrm{cl}(J)-c',$$
where $\mathrm{cl}(-)$ denotes the topological closure. Since both $I$ and $J$ are open intervals defined over $\mathbb{R}$, we obtain that $I\cap \mathbb{R}_{\rm alg}\subset I\subset J-c'$. Let us prove that
$$
\mathrm{st}(f(x+c)+d))=f(x+c')+d'
$$
and so $f(x+c')+d'=h(x)$ for all $x\in I\cap \mathbb{R}_{\rm alg}$.
If the latter holds true, then we clearly obtain that 
$$h(x)=f(x+c')+d'$$
for all $x\in I\subset \mathbb R$, since both $f_{|J\cap \mathbb R}:J\cap \mathbb R\to \mathbb R$ and $h:I\to\mathbb R$ are continuous, as required. 

Fix $x_0\in I\cap \mathbb{R}_{\rm alg}$ and let $U$ be a semialgebraic open neighborhood of $f(x_0+c')+d'$ defined over $\mathbb{R}$. We must show that $f(x_0+c)+d\in U$. Consider the semialgebraic map
$$\tilde h:(J-x_0)\times \bar R \to \bar R, \ (y,z)\mapsto f(x_0+y)+z$$
which is continuous and defined over $\mathbb{R}_{\mathrm{alg}}$. Thus, $\tilde h^{-1}(U)$ is an open semialgebraic set defined over $\mathbb{R}$. Since $(c',d')\in \tilde h^{-1}(U)$ and $\mathrm{st}(c,d)=(c',d')$ we deduce that $(c,d)\in  \tilde h^{-1}(U)$ and so $f(x_0+c)+d\in U$.
\end{substeps}	
\medskip

\begin{substeps} Assume $c\in \bar R$ with $\mathbb R<c$. In this case  we may assume that $J$ is an interval of the form $(j,+\infty)$ for some $j\in \mathbb{R}_{\mathrm{alg}}$. If $I(\bar R)$ denotes the $\bar R$-points of $I$, we can consider the $\bar R$-semialgebraic function
$$
g:I(\bar R)\to  \bar R, \ x\mapsto g(x):=f(x+c)+d
$$
which is also strictly monotone and $\mathcal{C}^1$. Recall that $g(x)\in \mathrm{Fin}_\mathbb{R}(\bar R)$ for every $x\in I\cap \mathbb{R}_{\mathrm{alg}}=I(\bar R)\cap \mathbb{R}_{\mathrm{alg}}$. Considering the derivatives functions, we have that $g'(x)=f'(x+c)$ for every $x\in I\cap \mathbb{R}_{\mathrm{alg}}$. 

We distinguish two cases depending on whether the limit of $f'(x)$ exists or not:

\begin{subsubsteps}{2}If $\lim_{x\mapsto +\infty}f'(x)=\pm\infty$, then for every $x\in I(\bar R)\cap \mathrm{Fin}_\mathbb{R}(\bar R)$ we have that $|g'(x)|=|f'(x+c)| >\mathbb R$. Otherwise, there are some $x_0\in I(\bar R)\cap \mathrm{Fin}_\mathbb{R}(\bar R)$ and some $n\in \mathbb{N}$ with $|g'(x_0)|=|f'(x_0+c)|<n$. But by assumption there must exists some non-negative $\delta\in \mathbb{R}$ such that for all $x\in \bar R$ with $x>\delta$ we have that $|f'(x)|>n$. So, since $x_0+c>\delta$, we get $|f'(x_0+c)|>n$, a contradiction.
	
Once we have seen that $|g'(x)|=|f'(x+c)| >\mathbb R$ for  every $x\in I(\bar R)\cap \mathrm{Fin}_\mathbb{R}(\bar R)$, we choose some $x_1,x_2\in I\cap \mathbb{R}_{\mathrm{alg}}$ with $x_1<x_2$. We have that $g(x_2), g(x_1)\in \mathrm{Fin}_\mathbb{R}(\bar R)$. Now, by the mean value theorem there is an element $\hat{x}\in I(\bar R)$ with $x_1<\hat{x}<x_2$ such that
\[
g(x_2)-g(x_1)=g'(\hat{x})\cdot (x_2-x_1).
\] 
Observe that $\hat{x}\in I(\bar R)\cap \mathrm{Fin}_\mathbb{R}(\bar R)$, as $x_1<x_2\in \mathbb R_{\rm alg}$. So, we readily deduce that $|g'(\hat{x})|\cdot (x_2-x_1) >\mathbb{R}$, contradicting the fact that $g(x_2)-g(x_1)\in \mathrm{Fin}_\mathbb{R}(\bar R)$.
\end{subsubsteps}	

\begin{subsubsteps}{2}Assume now that $\lim_{x\mapsto +\infty}f'(x)=\ell$ for some $\ell\in \mathbb{R}_{\mathrm{alg}}$. Let us show that for every $x\in I(\bar R)\cap \mathrm{Fin}_\mathbb{R}(\bar R)$ we have that $g'(x)\in\mathrm{Fin}_\mathbb{R}(\bar R)$ and $\mathrm{st}(g'(x))=\ell$. Indeed, for any $\epsilon\in \mathbb{R}_{>0}$ there is some non-negative $\delta_\epsilon\in \mathbb{R}$ such that for every $x\in J=(j,+\infty)$ with $x>\delta_{\epsilon}$ we have $|f'(x)-\ell|<\epsilon$. Since for every $x\in I(\bar R)\cap \mathrm{Fin}_\mathbb{R}(\bar R)$ and for every $\epsilon \in \mathbb R_{>0}$ we have that $x+c>\delta_\epsilon$, we deduce that
\[
|g'(x)-\ell|=|f'(x+c)-\ell|<\epsilon.
\]
Thus $g'(x)\in \mathrm{Fin}_\mathbb{R}(\bar R)$ for every $x\in I(\bar R)\cap \mathrm{Fin}_{\mathbb R}(\bar R)$  and we obtain that $\mathrm{st}(g'(x)) = \ell$, as claimed.
	
Next, fix some $x_1\in I\cap \mathbb{R}_{\text{alg}}$. By the mean valued theorem, for every $x_2\in I\cap \mathbb{R}_{\text{alg}}$ with $x_1<x_2$ there is some element $\hat{x}\in I(\bar R)$ such that $x_1<\hat{x}<x_2$ and
\[
g(x_2)-g(x_1)=g'(\bar{x})\cdot (x_2-x_1).
\]
Thus, applying the standard map we get $\mathrm{st}(g(x_2))=\mathrm{st}(g(x_1))+\ell\cdot (x_2-x_1)$. So
\[
h(x_2)=h(x_1)+\ell\cdot (x_2-x_1).
\]
It then follows that the function $k:=h(x)-\ell \cdot x\in \mathbb{R}$ is constant for every $x\in I\subset \mathbb{R}$, and so  $h(x)=\ell\cdot  x+k$. Finally, we consider the semialgebraic function $\tilde{f}:\bar R\rightarrow \bar R$ given by $x\mapsto \ell \cdot x$, which is clearly defined over $\mathbb{R}_{\mathrm{alg}}$. Hence, we obtain that  $h(x)=\tilde{f}(x+0)+k$ for all $x\in I$, as required.
	
\end{subsubsteps}	
\end{substeps}	

\begin{substeps} Assume finally that $c\in\bar R$ with $c<\mathbb R$. This case is a straightforward adaptation of the latter case. This finishes the proof.
\end{substeps}	
\end{proof}

\begin{remark}\label{r:extreal}\leavevmode
\begin{enumerate}
	\item Note that if $X\subset \mathbb{R}^m_{\rm alg}$ is an externally semialgebraic set then there is a saturated real closed field $R$ and a semialgebraic subset $Y\subset R^m$ such that $X=Y\cap \mathbb{R}^m_{\rm alg}$. Since we can assume that $R$ contains the real field $\mathbb R$, and types over $\mathbb{R}$ are definable, we deduce that $Z:=Y\cap \mathbb{R}^m$ is a real semialgebraic set with $X=Z\cap \mathbb{R}^m_{\rm alg}$. Thus, externally semialgebraic sets of $\mathbb{R}^m_{\rm alg}$ are Boolean combinations of sets as in the statement of Proposition \ref{prop:realalg}.
	\item The proof of Proposition \ref{prop:realalg} goes through for any o-minimal structure $\mathcal{R}_{\rm alg}$ whose universe is the field of real algebraic numbers $\mathbb{R}_{\rm alg}$  and such that it has an elementary extension $\mathcal{R} \succeq \mathcal{R}_{\rm alg}$  whose universe is the real field $\mathbb{R}$.
\end{enumerate} 
\end{remark}

Finally, we prove:

\begin{cor}\label{C:Alg} Let $\mathbb{R}_{\rm alg}$ be the field of real algebraic numbers and let $G=(\mathbb{R}^2_{\rm alg},+)$. The continuous map $\Lambda : S_G^{\rm ext}(\mathbb{R}_{\rm alg}) \to E(S_G(\mathbb{R}_{\rm alg}))$ is not an isomorphism. 	
\end{cor}
\begin{proof}By Theorem \ref{C:Equiv-Ellis-ext} it is enough to show that the externally definable set 
$$\{(x,y)\in  \mathbb{R}_{\rm alg}^2 \ | \ y<\pi \cdot x\}$$
is not $d$-definable. Otherwise, by Proposition \ref{prop:realalg} there is an open interval $I$ of $\mathbb{R}$ and there are a semialgebraic map $f:J\rightarrow \mathbb{R}$ defined over $\mathbb{R}_{\mathrm{alg}}$ and some elements $c,d\in \mathbb{R}$ such that $\pi\cdot x=f(x+c)+d$ for all $x\in I$. In particular, setting $z=x+c\in I+c$ we readily obtain
$$
f(z)=\pi\cdot (z-c)-d=\pi \cdot z-(\pi \cdot c+d).
$$
However, evaluating in any two distinct real algebraic numbers of $I+c$ we deduce that $\pi \in 	\mathbb{R}_{\text{alg}}$, a blatant contradiction.\end{proof}

\begin{remark}\label{defcompact}It is possible to refine the above result in order to find a definably compact example. Indeed, consider $G:=(\mathbb{R}_{\text{alg}}^2,+)$ and $G_1=\big([0,1), +_1\big)^2$  where $[0,1)\subset \mathbb{R}_{\text{alg}}$ and $+_1$ denotes the sum mod $1$. Note that the underlying definable set of $G_1$ is a definable subset of $G$, and therefore the (externally) definable subsets of $G_1$ are (externally) definable subsets of $G$. Note also that $G_1$ is definably compact and isomorphic to the quotient of $G$ by $\Z \times \Z$.
	
	Let $X\subset G_1$ be an externally definable subset such that $X$ is a Boolean combination of $d$-definable in the sense of $G_1$. We \emph{claim} that $X$ is a Boolean combination of $d$-definable in the sense of $G$. 
	
	To ease the reading, we denote by $M$ the field $\mathbb{R}_{\rm alg}$ and let $\bar M$ be an $|M|^+$-saturated elementary extension of $M$. Assume first that $C\subset G_1$ is a $d$-definable in the sense of $G_1$, and let us show that $C$ is a Boolean combination of $d$-definable sets in the sense of $G$. By definition, $C=(Z+_1 y)\cap M$ for some $M$-definable set $Z\subset G_1({\bar M})$ and $y\in G_1({\bar M})$. We show for $F:=\big(\Z\times \Z\big)\cap (-2,1)^2$ that 
	$$
	C=\left( \big((Z+y)\cap M\big)+F \right)\cap [0,1)^2.
	$$
	It is enough to prove that the left hand set is contained in the one of the right hand side, and note that $Z+y\subset [0,2)^2$. For every $x\in C$ there is $x_0\in (Z+y)\cap M$ such that $x-x_0\in \mathbb{Z}\times \mathbb{Z}$. Note that $x-x_0\in (-2,1)^2$ and so $x\in [\big((Z+y)\cap M\big)+F]$, as required. On the other hand, since $\big((Z+y)\cap M\big)+F=\big((Z+F)+y\big) \cap M$ where $Z+F$ is $M$-definable, we deduce that the right term of the equality is a Boolean combination of $d$-definable sets of $G$.
	
	For the general case, suppose $X=\bigcup_i \bigcap_j C_{ij}$ where each $C_{ij}$ is a $d$-definable subset of $G_1$. For each $i$ and $j$, let $Z_{ij}\subset G({\bar M})$ be an $M$-definable set and $y_{ij}\in G({\bar M})$ such that $C_{ij}:=(Z_{ij}+_1 y_{ij})\cap M$. Define for each $i$ and $j$ the set
	$$\widetilde{C}_{ij}:=\left(\big((Z_{ij}+y_{ij})\cap M\big)+F \right)\cap [0,1)^2,$$
	which by the above paragraph is a Boolean combination of $d$-definable subsets of $G$. It is easy to check that 
	$$X=\bigcup_i \bigcap_j {\widetilde C}_{ij},$$
	and therefore we deduce $X$ is also a Boolean combination of $d$-definable subsets of $G$, which finishes the proof of our claim.
	
	Finally, we deduce that the continuous map $\Lambda : S_{G_1}^{\rm ext}(\mathbb{R}_{\rm alg}) \to E(S_{G_1}(\mathbb{R}_{\rm alg}))$ is not an isomorphism. For, by Corollary \ref{C:Equiv-Ellis-ext} and the claim above it is enough to show that the externally definable set
	$$\{(x,y)\in \mathbb{R}^2_{\rm alg}: (0<x<1) \wedge 0<y<\tfrac{\pi}{4}y\}$$
	is not a Boolean combination of $d$-definable sets of $G=( \mathbb{R}^2_{\rm alg},+)$, which follows by a similar argument as in Corollary \ref{C:Alg}.
\end{remark}

\end{document}